\documentclass[11pt]{amsart}
\usepackage{amssymb}
\usepackage{amsmath}
\usepackage{amsfonts,latexsym,amstext}

\usepackage [T1]{fontenc}
\usepackage [latin1]{inputenc}
\usepackage{color}

\newtheorem{theorem}{Theorem}[section]

\newtheorem{lemma}[theorem]{Lemma}
\newtheorem{proposition}[theorem]{Proposition}
\newtheorem{corollary}[theorem]{Corollary}

\newcommand{\C}{\mathbb{C}}
\newcommand{\D}{\mathbb{D}}

\newcommand{\A}{\mathcal{A}}

\def\v{\overset\vee}

\def\M{\mathcal M}

\begin{document}


\title{Cluster values for algebras of analytic functions}

\author{Daniel Carando, Daniel Galicer, Santiago Muro, Pablo Sevilla-Peris}



\address{DEPARTAMENTO DE MATEM\'ATICA - PAB I,
FACULTAD DE CS. EXACTAS Y NATURALES, UNIVERSIDAD DE BUENOS AIRES, (1428) BUENOS AIRES, ARGENTINA AND CONICET} \email{dcarando@dm.uba.ar} \email{dgalicer@dm.uba.ar} \email{smuro@dm.uba.ar}

\address{INSTITUTO UNIVERSITARIO DE MATEM\'ATICA PURA Y APLICADA, UNIVERSITAT POLIT\`ECNICA DE DE VAL\`ENCIA, CMNO VERA S/N, 46022, VALENCIA, SPAIN}
\email{psevilla@mat.upv.es}



\thanks{This work was partially supported by projects CONICET PIP 06
24, ANPCyT PICT 2011-1456, ANPCyT PICT 11-0738,
UBACyT 20020130300057BA, UBACyT20020130300052BA, UBACyT 20020130100474BA and MINECO and FEDER Project MTM2014-57838-C2-2-P.}

\begin{abstract}
The {\it Cluster Value Theorem} is known for being a weak version of the classical \emph{Corona Theorem}.
Given a Banach space $X$, we study the \emph{Cluster Value Problem} for the ball algebra $A_u(B_X)$, the
Banach algebra of all uniformly continuous holomorphic functions on the unit ball
$B_X$; and also for the Fr\'echet algebra $H_b(X)$ of holomorphic functions of bounded type on $X$ (more generally, for $H_b(U)$, the algebra of holomorphic functions of bounded type on a given balanced open subset $U \subset X$). We show that Cluster Value Theorems hold for all of these algebras whenever the dual of $X$ has the bounded approximation property. 
These results are an important advance in this problem, since the validity of these theorems was known only for trivial cases (where the spectrum
is formed only by evaluation functionals) and for the infinite dimensional Hilbert space.

As a consequence , we obtain weak \emph{analytic Nullstellensatz theorems} and several structural results for the spectrum of these algebras. 
\end{abstract}
\maketitle

\section*{Introduction}

In the 1940's a change of perspective in complex analysis emerged. S. Kakutani started a systematic  study, from a Banach algebra point of view, of the space  of bounded holomorphic functions on the open unit disk of the complex plane, $H^\infty(\D)$.

A prominent group of mathematicians (Singer, Wermer, Kakutani, Buck, Royden, Gleason, Arens and Hoffman) made several contributions in this line, and in 1961 published a very interesting paper \cite{Schark61} under the
fictitious name of I.~J.~Schark.
They showed that for each function $f\in H^\infty(\D)$, the set of evaluations $\varphi(f)$ of elements on the maximal ideal space lying in the fiber over a point $z\in\partial\D$ coincides with the cluster set of $f$ at $z$ (i.e., the set of all limits of values of $f$ along nets converging to $z$).
This result, called the \emph{Cluster Value Theorem}, was a sort of predecessor (and a weak version) of the famous Corona Theorem due to Carleson \cite{Carleson62}, which states that the characters given by evaluations on points of $\D$ are dense on the maximal ideal space of $H^\infty(\D)$.

Several domains where the Corona Theorem fails are known \cite{sibony1974prolongement} but it is unknown whether there is \emph{any} domain in $\C^n$, for $n \geq 2$, for which the  Corona Theorem holds. On the other hand, up to our knowledge, no domain is known for which the Cluster Value Theorem is not true.

In the context of infinite dimensional complex analysis, characterizing the cluster set is a non-trivial task, even for interior points. This fact led, in the last years,  many mathematicians in the area to study the \emph{Cluster Value Problem} for different algebras of analytic functions on infinite dimensional Banach spaces. Among their contributions, they showed positive results on the Cluster Value Theorem for the algebra $H^\infty(B_X)$ of bounded holomorphic functions on $B_X$, the
open unit ball of the Banach space $X$ \cite{AroCarGamLasMae12,johnson2015cluster,johnson2014cluster};  the ball algebra $A_u(B_X)$ of all uniformly continuous holomorphic functions on
$B_X$ \cite{AroCarGamLasMae12}, and the algebra $H_b(X)$ of holomorphic functions of bounded type on $X$ \cite{AroCarLasMae} (note that the first two, endowed with the norm $\Vert f \Vert = \sup_{x \in B_{X}} \vert f(x) \vert$ are Banach algebras and the last one, with the topology of uniform convergence on bounded sets, is a Fr\'echet algebra). These results were given for very specific Banach spaces. Namely, for the algebra $H^\infty(B_X)$, the known cases (besides finite dimensional spaces) are $c_0$ and $C(K)$, the space of null sequences and the space of continuous functions over a dispersed compact Hausdorff space $K$ respectively \cite{AroCarGamLasMae12,johnson2014cluster}.
For the algebras $A_u(B_X)$ and $H_b(X)$, the validity of the Cluster Value Theorem was known only for trivial cases (where the spectrum is formed only by evaluation functionals) and for the infinite dimensional Hilbert space \cite{AroCarGamLasMae12,AroCarLasMae}. Apart from these, there is (up to our knowledge) no general positive result, although there are some for the origin: in \cite[Theorem 3.1]{AroCarGamLasMae12} it is shown that the Cluster value Theorem for $A_u(B_X)$ holds at $0$ whenever $X$ is a
Banach space with a shrinking 1-unconditional basis (or, more generally, with a shrinking reverse monotone
FDD \cite[Corollary 2.4]{johnson2014cluster}). Analogous results (at the origin) were given in \cite{AroCarLasMae} for $H_b(X)$.

A main tool in \cite[Theorem 4.1]{AroCarGamLasMae12} to show that the Cluster Value Theorem for $A_u(B_{\ell_2})$ holds at every point of the ball is the existence of transitive analytic automorphisms (with some necessary properties) on $B_{\ell_2}$. Unfortunately, these ideas cannot be adapted to arbitrary Banach spaces (for example, to $\ell_p$ for $p\ne 2$), due to the lack of these automorphisms.

The purpose of this note is to exhibit that the Cluster Value Theorem holds for every Banach space $X$ whose dual has the bounded approximation property  for the algebras $A_u(B_X)$, $H_b(X)$ and $H_b(U)$ for any balanced open set $U$ of $X$ (see Theorems \ref{main1}, \ref{cvt-hb-bap} and \ref{CVT_HbU}).
This provides answers to several open problems in the area (in particular, it gives answers to the questions posed in \cite[Remark 2.5]{johnson2015cluster} and also \cite[Questions 7 and 8]{carando2012spectra}).

The proof of the main theorem does not rely on automorphisms on the ball of the space. Instead, we show in Theorem~\ref{equivalencia CVT} that if the Cluster Value Theorem holds for one of the algebras $A_u(B_X)$, $H_b(X)$ or $H_b(U)$ then it also holds for any of the others. Then, we exploit the fact that one can ``freely move'' in the whole space $X$ to show that the Cluster Value Theorem holds in $H_b(X)$. In some sense, we follow the spirit of \cite{AroColGam91}, where $H_b(X)$ is studied in order to describe the spectrum of some uniform algebras on $B_X$. Here, we go one step further in this line, since we obtain concrete results for $A_u(B_X)$ from studying the same problem in  $H_b(X)$.

We also present some consequences of the Cluster Value Theorem, as weak \emph{analytic Nullstellensatz} theorems on the corresponding algebras and some properties of their spectra.

The article is organized as follows. In Section \ref{Au} we state the Cluster Value Theorem and relate it with the Corona Theorem for $A_u(B_X)$. In Section \ref{Hb} we state and prove the Cluster Value Theorem for $H_b(X)$; this will give as a consequence the corresponding theorem for $A_u(B_X)$. In Section \ref{aplicaciones}  we use the Cluster Value Theorems  to obtain several results for the spectrum of the different algebras considered.

We refer the reader to \cite{AlKa06} and \cite{LinTza73} for the general theory of Banach spaces and to \cite{Din99} and \cite{Muj86} for the theory of polynomials and holomorphic functions on infinite dimensional spaces. We also refer to the survey \cite{carando2012spectra} for several recent developments
on the research of the spectra of algebras of analytic functions.

\section{The Cluster Value Theorem for $A_u(B_X)$}\label{Au}

We begin with some background and basic definitions.
Given a Banach or Fr\'{e}chet algebra $\A$ of holomorphic
functions on an open subset $U$ of a complex Banach space
$X$, we write $\M(\A)$ for the spectrum of $\A$, i.e. the set of all (continuous and non-zero) complex valued
homomorphisms on $\A$.
This is a subspace of the dual $\A^{*}$, which is endowed with the weak-star
topology $\sigma(\A^*,\A)$. If $\A$ contains $X^*$ then $\M(\A)$ is fibered over $X^{**}$, and we have the following commutative diagram:
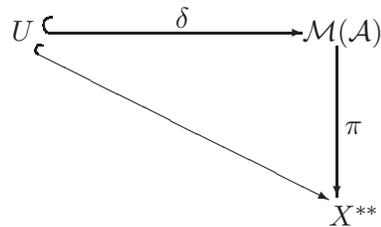
\vspace{0.5cm}
\begin{center}
\begin{picture}(130,90)
\put(0,85){$U$} \put(110,85){$\M(\A)$} \put(120,15){$X^{\ast\ast}$}
\put(15,88){\vector(1,0){95}} \qbezier(15,88)(12,88)(12,91) \qbezier(12,91)(12,94)(15,94)
\put(123,83){\vector(0,-1){57}} \put(10,80){\vector(2,-1){110}}
\qbezier(10,80)(8.4,80.8)(9.2,82.4) \qbezier(9.2,82.4)(10,84)(11.6,83.2) \put(62,90){$\delta$}
\put(126,50){$\pi$}
\end{picture}
\end{center}
 where $\delta$ is the point evaluation mapping and $\pi$ is defined by $\pi
(\varphi)=\varphi |_{X^{\ast}} \in X^{\ast\ast}$, $\varphi \in \M(\A)$.

\bigskip
Let $\mathcal A$ be $A_u(B_X)$ or $H^\infty(B_X)$.

We say that the Corona Theorem holds for $\A$ whenever
\begin{equation} \label{coronastatement}
\M(\A)= \overline{\{\delta_x\colon
x\in B_X\}}^{\sigma(\A^*, \A)}.
\end{equation}
In other words, the evaluations on
points of $B_X$ form a dense set in $\M(\A)$.

The fiber of $\M(\A)$ over a point $z\in \overline{B_{X^{**}}}$ is the set 
\[
\M_z(\A) = \{ \varphi\in \M(\A) \, \colon \,   \pi(\varphi)=z \} \,.
\]
An obvious reformulation of the Corona Theorem  in terms of fibers is the following: for each $z\in \overline{B_{X^{**}}}$, we have
\[
\M_z(\A)\cap  \overline{\{\delta_x\colon
x\in B_X\}}^{\sigma(\A^*, \A)}= \M_z(\A).
\]
Inspired by this, we say that $\A$ satisfies the \emph{weak Corona Theorem} if for each $f\in \A$ and $z\in \overline{B_{X^{**}}}$ we have
\begin{equation}\label{weakCT}
\hat f\big(\M_z(\A)\cap  \overline{\{\delta_x\colon
x\in B_X\}}^{\sigma(\A^*, \A)}\big)= \hat{f}\big(\M_z(\A)\big),
\end{equation}
where $\hat f:\M(\mathcal{A}) \to \C$ is the Gelfand transform
of $f$, given by $\hat f(\varphi)=\varphi(f)$.

Note that if $\varphi$ belongs to $\M_z(\A)\cap  \overline{\{\delta_x\colon
x\in B_X\}}^{\sigma(\A^*, \A)}$, then there is a net $(x_{\alpha}) \subset B_X$ such that $\delta_{x_{\alpha}} \overset{\sigma(\A^* ,\A)}{\longrightarrow}  \varphi$, which necessarily satisfies $x_{\alpha} \overset{\sigma(X^{**},X^*)}{\longrightarrow} z$. Since $\M(\A)$ is compact, standard arguments show that the set at the left in \eqref{weakCT} coincides with the so called \emph{cluster set} of $f$ at $z$, which we define now.

For a fixed function $f\in \A$, and  $z \in \overline{B_{X^{**}}}$, the \textit{cluster set of $f$
at $z$} is the set $Cl(f,z)$ of all limits of values of $f$ along nets in $B_X$ weak-star converging to $z$; that is
\begin{equation} \label{cluster}
\begin{split}
 Cl(f,z)=\{\lambda\in\mathbb C\,:\, \textrm{there exists a net } & (x_\alpha)\subset B_{X} \\
 & \textrm{ such that }x_\alpha\overset{\sigma(X^{**},X^*)}{\longrightarrow}z,\textrm{ and } f(x_\alpha)\to\lambda\}.
\end{split}
\end{equation}

We always have (see \cite[Lemma 2.2]{AroCarGamLasMae12}) that
\begin{equation}\label{vale siempre Au}Cl(f,z)\subset \hat f(\M_{z}(\A)).
\end{equation}
Whenever the inclusion \eqref{vale siempre Au} is an equality, we say
that the \emph{Cluster Value Theorem holds in $\A$ at $z$.} If this happens for every $z \in \overline{B_{X^{**}}}$, we say simply that the Cluster Value Theorem holds in $\A$ or, equivalently, that the weak Corona Theorem \eqref{weakCT} holds in $\A$.

We now present our contribution to the Cluster Value Problem for $A_u(B_X)$.

\begin{theorem}\label{main1}
 Let $X$ be a Banach space whose dual has the bounded approximation property, then
$$Cl(f,z)=\hat f (\M_{z}(A_u(B_X)))\quad \text{for every }f\in A_u(B_X)\text{ and }z\in \overline{B_{X^{**}}}.$$
\end{theorem}

This theorem is a direct consequence of Theorem~\ref{cvt-hb-bap}, which is the analogous result for $H_b(X)$,  and Theorem~\ref{equivalencia CVT}, which shows the equivalence of the Cluster Value Theorems for $A_u(B_X)$ and $H_b(X)$.

\medskip

Suppose again that $\mathcal A$ is equal to $A_u(B_X)$ or $H^\infty(B_X)$. It is well known that the Corona Theorem holds for $\A$ if and only if whenever  $f_1, \ldots, f_n
\in \A$ satisfy  $|f_1| + \cdots + |f_n| \ge \varepsilon > 0$ on $B_X$,  there
exist  $g_1, \ldots , g_n \in \A$ such that  $f_1 g_1 + \cdots + f_n
g_n = 1$. This can be seen as an \emph{analytic Nullstellensatz} theorem for the algebra $\A$.

We denote by $A(B_X)$ the closed subalgebra of $A_u(B_X)$ generated by $X^*$, that is, every element in $A(B_X)$ can be uniformly approximated by finite type polynomials.
It was shown in \cite[Lemma 2.3]{AroCarGamLasMae12} that the Cluster Value Theorem  for $A_u(B_X)$ is equivalent to the following weak form of the analytic Nullstellensatz theorem for the ball algebra. Therefore, as a consequence of Theorem \ref{main1}, we obtain:

\begin{corollary}\label{weak-null}
Let $X$ be a Banach space whose dual  has the bounded approximation property. Given  $f_1, \ldots, f_{n-1} \in A(B_X)$ and $f_n \in A_u(B_X)$ such that  $|f_1| + \cdots  +
|f_n| \ge \varepsilon > 0$ on $B_X$,  there exist  $g_1, \ldots , g_n \in
A_u(B_X)$ such that  $f_1 g_1 + \cdots + f_n g_n = 1$.
\end{corollary}

\section{The Cluster Value Theorem for $H_b(X)$}\label{Hb}

In this section we state and prove the Cluster Value Theorem for the space of entire functions of bounded type, together with the equivalence of the validity on the other algebras. We obtain Theorem~\ref{main1}  as a consequence.

For an open subset $U\subset X$, a $U$-bounded set is a bounded set $A\subset U$ whose distance to the boundary of $U$ is positive.

   The space of all complex valued holomorphic functions on $U$ which are bounded on $U$-bounded sets is denoted by $H_b(U)$. The space $H_b(U)$ is a Fr\'echet algebra when it is endowed with the topology of uniform convergence on $U$-bounded sets. 
Thus, for each homomorphism $\varphi\in\M(H_b(U))$, there exists a $U$-bounded subset $A$ such that
\begin{equation}\label{varphi<A}
  |\varphi(f) | \leq \Vert f \Vert_{A}, \quad\textrm{ for every }f\in H_b(U),
\end{equation}
where $\|f\|_{A}$ is the supremum of $|f|$ over the set $A$. We will write $\varphi\prec A$ when \eqref{varphi<A} happens.
Given $A \subset U$ we define

\[
\M_A(H_b(U)) = \{ \varphi\in \mathcal M(H_b(U))\,:\, \varphi\prec A \} \,.
\]
Also, given $z \in X^{**}$ we also define the \emph{$A$-fiber of $\mathcal M(H_b(U))$ over $z$} as the set
  $$
  \mathcal M_{z,A} (H_b(U)) =\{\varphi\in \mathcal M(H_b(U))\,:\, \pi(\varphi)=z\textrm{ and }\varphi\prec A\} = \M_A(H_b(U)) \cap \M_{z}(H_b(U)),
  $$
where, of course, $\M_{z}(H_b(U)) =\{\varphi\in \mathcal M(H_b(U))\,:\, \pi(\varphi)=z \}$.
We will often simply write $\mathcal M_{z,A}$ instead of $ \mathcal M_{z,A} (H_b(U)) $.

 Let $U\subset X$ be an open set, $f\in H_b(U)$  and  $z\in X^{**}$. For $A\subset U$, we define,

following \eqref{cluster}, the cluster set $Cl_A(f,z)$ as the set of all limits of values of $f$
along nets in $A$ converging weak-star to $z$, that is,
 $$
  Cl_A(f,z)=\{\lambda\in\mathbb C\,:\, \textrm{there exists a net }(x_\alpha)\subset A\textrm{ such that }x_\alpha\overset{\sigma(X^{**},X^*)}{\longrightarrow}z,\textrm{ and } f(x_\alpha)\to\lambda\}.
  $$
Note that $Cl_A(f,z)$ is empty if (and only if) $z$ does not belong to the weak-star closure of $A$.
Also, in the particular
case that $z$ actually belongs to $A$,  $f(z)$  belongs to
$Cl_A(f,z)$.
The cluster set  $Cl_A(f,z)$ can be seen as the intersection of the closures of $f(V \cap A)$, where $V$ ranges over any basis $\mathcal{V}$ of the $\sigma(X^{**}, X^{*})$ neighborhoods of $z$.

It is not difficult to see that, as in Equation \eqref{vale siempre Au}, we always have the inclusion
\begin{equation}\label{vale siempre}Cl_A(f,z)\subset \hat f(\M_{z,A}),
\end{equation} where $\hat f:\M(H_b(U)) \to \C$ is the Gelfand transform
of $f$ given by $\hat f(\varphi)=\varphi(f)$.
%
Therefore, we are
interested in the reverse inclusion. When the equality holds for every $f\in H_b(U)$, every $z\in U$ and every $U$-bounded set $A$,  we say
that the {\it Cluster Value Theorem} holds for $H_b(U)$. Note that the Cluster Value Theorem for $H_b(X)$ considered in \cite{AroCarLasMae}, which is stated as $Cl_r(f,z)=\hat f (\M_{z,r}) \quad\textrm{for every } f\in H_b(X),\, z\in X^{**} \textrm{ and }r>0,$ is implied by our version, just taking $A=rB_X$.

 The Corona Theorem holds for $H_b(X)$ if the set of evaluations at points of $X$ is dense in $\M(H_b(X))$. This holds trivially if $X$ is
finite dimensional since in this case $\M(H_b(X))=\{\delta_x\colon x\in X\}$. It is also easy to see that the Corona Theorem for $H_b(X)$ is true when the algebra generated by $X^*$ is dense in $H_b(X)$ (for example, if $X=c_0$ or if $X=T^*$, the original Tsirelson space), in which case the spectrum is identified with $X^{**}$. Moreover, it was proved in \cite[Corollary 4.7]{AroColGam91} that, for any Banach space $X$, any homomorphism $\varphi$ restricted to the set of analytic polynomials on $X$ is the limit of a net of evaluations at points of $X$. It is an open question whether the Corona theorem holds for $H_b(X)$ on any Banach space $X$ for which the spectrum does not coincide with $X^{**}$.
As we did in the previous section, it is easy to see that the Cluster Value Theorem for $H_b(X)$ is equivalent to a weak version of the bounded Corona Theorem for $H_b(X)$, which can be stated as follows:
let $A\subset X$ be bounded, then any homomorphism $\varphi\in \M(H_b(X))$ such that $\varphi\prec A$ can be approximated by a net of evaluations at points of $A$, or equivalently, for each bounded set $A\subset X$ and each $z\in X^{**}$, we have
$$
\M_{z,A}(H_b(X))\cap  \overline{\{\delta_x\colon
x\in A\}}^{\sigma(H_b(X)^*, H_b(X))}= \M_{z,A}(H_b(X)).
$$
Of course, since the bounded Corona Theorem for $H_b(X)$ is stronger than the Corona Theorem for $H_b(X)$, it is unknown to be true unless $\M(H_b(X))=\{\delta_x\colon x\in X^{**}\}$.

The next theorem is the main result of this section.

\begin{theorem}\label{cvt-hb-bap}
Let $X$ be a Banach space whose dual  has the bounded approximation property. Then
$Cl_A(f,z)=\hat f (\M_{z,A}) \quad\textrm{for every } f\in H_b(X),\, z\in X^{**} \textrm{ and every  bounded set }A.$
\end{theorem}

We need some lemmas to prepare the proof of Theorem \ref{cvt-hb-bap}. The first one shows that a formally weaker version of the Cluster Value Theorem is actually equivalent. This is the main tool that will allow us to show the validity of the Cluster Value Theorem for Banach spaces whose duals have the bounded approximation property. As a bridge between the two versions of the Cluster Value Theorem we have a weak version of the analytic Nullstellensatz theorem for $H_b(U)$.

 \begin{lemma}\label{CVTT}
   Let $X$ be a Banach space and $U\subset X$ an open set. The following are equivalent:
   \begin{itemize}
   \item[$i)$] for each $U$-bounded set $A$, each $z \in X^{**}$ and each $f\in H_b(U)$, there exists a $U$-bounded set $\tilde A$ such that
    $$
    \hat f \big( \mathcal M_{z,A} )\subset Cl_{\tilde A}(f,z).
    $$
   \item[$ii)$] Given  $f_1,\dots,f_n\in H_b(U)$, with $f_j$ approximable by finite type polynomials on $U$-bounded sets for $j=1,\dots,n-1$, such that,
    $$
    \inf_{x\in A}|f_1(x)|+\dots+|f_n(x)|>0\quad\textrm{ for each }U\textrm{-bounded set }A,
    $$
    there exist $g_1,\dots,g_n\in H_b(U)$, such that
    $$
    f_1g_1+\dots+f_ng_n=1.
    $$
          \item[$iii)$] for each $U$-bounded set $A$, each $z \in X^{**}$ and each $f\in H_b(U)$ we have
    $$
    \hat f \big( \mathcal M_{z,A}\big) = Cl_{ A}(f,z).
    $$
 \end{itemize}
 \end{lemma}

\begin{proof}

$i) \Longrightarrow ii)$
Let $f_1,\dots,f_n$ be as in $ii)$. First we see that $\hat{f_1}, \dots, \hat{f_n}$ do not have a common zero in $\M(H_b(U))$. Indeed, suppose there exists $\varphi \in \M(H_b(U))$  common zero of $\hat{f_1}, \dots, \hat{f_n}$. Let $A$ be a $U$-bounded set such that $\varphi \prec A$ and set $z:=\pi(\varphi)$ (thus, $\varphi \in \mathcal M_{z,A}$). By $i)$, there exists a $U$-bounded set $\tilde A$ such that $\varphi(f_n) \in Cl_{\tilde A}(f_n,z)$. In other words, there is a net $(x_\alpha)  \in \tilde A$ such that $x_\alpha\overset{\sigma(X^{**},X^*)}{\longrightarrow}z$ and $f_n(x_\alpha) \to \varphi(f_n)=0$.
Note that $0 = \hat f_j (\varphi) = f_j(z) = \lim f_j(x_\alpha)$ for every $j=1, \dots, n-1$ (since each $f_j$ is approximable).
Therefore,
$$\inf\{|f_1(x)|+\dots+|f_n(x)|\,:\, x\in \tilde A\}=0,$$
which is a contradiction.

Now, consider $\sigma(f_1,\dots,f_n)$  the joint spectrum of $f_1, \dots, f_n$, i.e., the set of all $(\lambda_1, \dots, \lambda_n) \in \mathbb C$ such that the elements $f_1 - \lambda_1, \dots, f_n - \lambda_n$ generate a proper ideal in $H_b(U)$. A theorem by  Arens (see \cite{arens1958dense} and also \cite[Proposition 32.13]{Muj86}) gives
$$
\sigma(f_1, \dots, f_n) = \{ \big(\varphi(f_1), \dots, \varphi(f_n) \big) : \varphi \in \M(H_b(U)) \}.
$$

We have just proved that $\hat f_1, \dots, \hat f_n$ do not have a common zero in $\M(H_b(U))$, thus $(0, \dots, 0)$ does not belong to $ \sigma(f_1, \dots, f_n)$. This proves the existence of $g_1,\dots,g_n\in H_b(U)$, such that $ f_1g_1+\dots+f_ng_n=1$.

$ii) \Longrightarrow iii)$
Suppose $iii)$ does not hold and take  a $U$-bounded set $A$, $z \in U$, $f\in H_b(U)$ and $\varphi \in  \mathcal M_{z,A}$ such that $\varphi(f)$ does not belong to $Cl_{ A}(f,z)$.
By changing $f$ by $f - \varphi(f)$ we can suppose, without loss of generality, that $0$ is not in  $Cl_{ A}(f,z)$. Then there exist $\varepsilon> 0$ and a $\sigma(X^{**},X^*)$-neighborhood $V= \{x \in A : \vert x_i^{*}(x) - z(x_i^{*}) \vert < \delta \mbox{ for } i=1, \dots, n-1 \} $ of $z$ such that $\vert f \vert > \varepsilon$ on $V\cap A$.
Define, for each $j=1, \dots, n-1$, $f_j := x_j^{*} - z(x_j^{*})$ (which is obviously approximable), then $\sum_{j=1}^{n-1} \vert f_j \vert > \delta$ on $A \setminus  V$. If $f_n : = f$, we have
$$
    \inf\{|f_1(x)|+\dots+|f_n(x)|\,:\, x\in A\}>0.
    $$
But $\hat f_1, \dots, \hat f_n$ vanish at $\varphi$, which is a contradiction.

$iii) \Longrightarrow i)$ is immediate.
\end{proof}

\begin{lemma} \label{igual en funcional}
 Let $T$ be a finite rank operator on $X$ and $z\in X^{**}$. Then for every $\varphi\in \M(H_b(X))$ with $\pi(\varphi)=z$ and $f\in H_b(X)$ we have
 $$
 \varphi(f)=\varphi(x\mapsto[f(x+T^{tt}z-Tx)]).
 $$
\end{lemma}
\begin{proof}
 Suppose $T=\sum_{k=1}^Nx_k^{*}\otimes x_k$. A repeated application of \cite[Lemma 1.5]{AroCarGamLasMae12} (note that, using the notation of \cite[Lemma 1.5]{AroCarGamLasMae12}, the hypothesis $x^*(e)=1$ is not needed for its proof) shows that there exist $g_k\in H_b(X)$, $k=1,\dots,N$, such that
 $$
 f(x)=f(x+T^{tt}z-Tx) + \sum_{k=1}^N(x_k^{*}(x)-z(x_k^{*}))g_k(x).
 $$
 The result follows because $\varphi(x\mapsto[x_k^{*}(x)-z(x_k^{*})])=0$.
\end{proof}

 \begin{lemma}\label{lema op finito}
  Let $X$ be a Banach space  whose dual  has the bounded approximation property. Let $z\in X^{**}$. For each $r>0$ there exists $R>0$ such that for each $\sigma(X^{**}, X^{*})$-neighborhood  $V$ of $z$, there is a finite rank operator $T$ on $X$ such that
  $$
  x+T^{tt}(z)-Tx \in V\cap RB_X \quad \textrm{ for every }x\in rB_X.
  $$
   \end{lemma}
\begin{proof}
If $X^{*}$ has the $\lambda$-approximation property, then there is a net $(T_\alpha)_{\alpha\in \Lambda	}$ of finite rank operators on $X$ such that: $\|T_\alpha\|\le \lambda$ for every $\alpha\in\Lambda$ and $T_\alpha^{*}x^{*}\to x^{*}$ for every $x^{*}\in X^{*}$ (see, for example, \cite[Propostion 3.5]{casazza2001approximation}).

Let $R:=\lambda\|z\|+(\lambda+1)r$. We may assume that $V=\{w\in X^{**}:\, |w(x_j^{*})-z(x_j^{*})|<\varepsilon,\, j=1,\dots,N\}.$ Let $\alpha\in\Lambda$ such that $\|T_\alpha^{*}(x_j^{*})-x_j^{*}\|<\frac{\varepsilon}{2}\min\{\frac1{r},\frac1{\|z\|}\}$. Then if $x\in rB_X$, it is easy to check that $x+T^{tt}(z)-Tx\in RB_X\cap V$.

\end{proof}

\medskip

\begin{proof}[Proof of Theorem~\ref{cvt-hb-bap}.]
Let $r>0$ and let $R>0$ be as in Lemma \ref{lema op finito}.
Using Lemma \ref{CVTT} it is enough to prove that if $0 \notin Cl_{RB_X}(f,z)$ then $0 \notin \hat f (\M_{z,{rB_X}})$.

Since $0 \notin Cl_{RB_X}(f,z)$ there is a $\sigma(X^{**}, X^{*})$-neighborhood  $V$ of $z$ and $\delta > 0$ such that $\vert f \vert > \delta$ on $V\cap RB_X$.
By Lemma \ref{lema op finito} there is a finite rank operator $T$ on $X$ such that
$ x+T^{tt}(z)-Tx \in V\cap RB_X$ for every $x\in rB_X$.
Then, the function $g(x):=f(x+T^{tt}z-Tx)$ satisfies $\vert g \vert > \delta$ on $rB_X$ and therefore it is invertible in $A_u(rB_X)$. Thus, $\varphi(g) \neq 0$ for every $\varphi \in\mathcal M(A_u(rB_X))$.
Now, note that $\M_{z,{rB_X}} \subset \M(A_u(rB_X))$. By Lemma \ref{igual en funcional} we have
$\varphi(f)=\varphi(g) \neq 0,$
and thus $0 \notin \hat f (\M_{z,{rB_X}})$.
\end{proof}

\bigskip

The following theorem shows the connection between the Cluster Value Theorem for  $H_b(X)$ and $A_u(B_X)$.

\begin{theorem}\label{equivalencia CVT}
Let $X$ be a Banach space. The following statements are equivalent:
\begin{itemize}
 \item[$i)$] For every bounded set $A\subset X$, every $z \in X^{**}$ and every $f\in H_b(X)$ we have
    $
    \hat f \big( \mathcal M_{z,A}(H_b(X))\big) = Cl_{ A}(f,z).
    $
 \item[$ii)$] For every  $z \in \overline{B_{X^{**}}}$ and every $f\in A_u(B_X)$ we have
    $
    \hat f \big( \mathcal M_{z}(A_u(B_X))\big) = Cl(f,z).
    $
\item[$iii)$] For every $z \in B_{X^{**}}$, every bounded set $A\subset B_X$ and every $f\in H_b(B_X)$ we have
    $
    \hat f \big( \mathcal M_{z,A}(H_b(B_{X}))\big) = Cl_{ A}(f,z).
    $
    \end{itemize}
\end{theorem}

\begin{proof}
$i) \Rightarrow ii)$: Let   $f\in A_u(B_X)$ and $\varphi \in \mathcal M\big(A_u(B_X)\big) \subset \mathcal M(H_b(X))$ such that $\pi(\varphi)=z$.
Fix $\varepsilon >0$ and $V$ a $\sigma(X^{**}, X^{*})$-neighborhood of $z$. To prove that $\hat f \big( \mathcal M_{z}( A_u(B_X))\big) \subset Cl(f,z)$, it is enough to show that there is $x \in V \cap B_X$ such that $\vert \varphi(f) - f(x) \vert < \varepsilon$. Thus,. Indeed, take $g \in H_b(X)$ such that
$\Vert f - g \Vert_{B_X} < \frac{\varepsilon}{3}$. By $i)$  we have
$$\hat g \big( \mathcal M_{z,B_X}\big) = Cl_{B_X}(g,z).$$ Then, there is $x \in V \cap B_X$ such that
$\vert \varphi(g) - g(x) \vert < \frac{\varepsilon}{3}$.
Now, since $\varphi \prec B_X$, we get
\begin{align*}
\vert \varphi(f) - f(x) \vert  & \leq \vert \varphi(f) - \varphi(g)  \vert + \vert \varphi(g) - g(x) \vert + \vert g(x) - f(x) \vert \\
& \leq \Vert f - g \Vert_{B_X} + \vert \varphi(g) - g(x) \vert + \Vert f - g \Vert_{B_X} \\
& \leq \varepsilon.
\end{align*}


$ii) \Rightarrow iii)$ Let $f\in H_b(B_X)$ and $\varphi \in M_{z,A}(H_b(B_X))$. We choose $r<1$ such that $A\subset rB_X$. Thus $f\in A_u(rB_X)$ and $\varphi\in \M_z(A_u(rB_X))$. It is easy to check  that if the Cluster Value Theorem holds for $A_u(B_X)$ then it also holds for $A_u(rB_X)$. Then, by $ii)$, there is a net $(x_{\alpha})_{\alpha \in \Lambda} \subset rB_X$ such that  $x_\alpha\overset{\sigma(X^{**},X^*)}{\longrightarrow} z$ and $f(x_\alpha) \to \varphi(f)$.
Therefore
$$
    \hat f \big( \mathcal M_{z,A}(H_b(B_{X})) )\subset Cl_{rB_X}(f,z).
$$
Now the result follows from Lemma \ref{CVTT}.

$iii) \Rightarrow i)$ Let $\varphi \in \M(H_b(X))$ and $f \in H_b(X)$ such that $\pi(\varphi)=z$ and $\varphi \prec A$ with  $A \subset rB_X$ for some $r>0$. Then $\varphi$ is continuous in the topology of $H_b(rB_X)$ and thus we can consider $\varphi \in \mathcal{ M}(H_b(rB_X))$ and $f \in H_b(rB_X)$. Again, if $iii)$ holds, then the Cluster Value Theorem holds for $H_b(rB_X)$.  Then there is a net $(x_{\alpha})_{\alpha \in \Lambda} \subset rB_X$ such that  $x_\alpha\overset{\sigma(X^{**},X^*)}{\longrightarrow} z$ and $f(x_\alpha) \to \varphi(f)$.
Therefore
$$
    \hat f \big( \mathcal M_{z,A}(H_b(X)) )\subset Cl_{rB_X}(f,z).
    $$
And again using Lemma \ref{CVTT}, the result follows.
\end{proof}

Recall that a set $U$ is balanced whenever $\lambda \cdot U \subset U$ for all $\vert \lambda \vert \leq 1$.
Since the entire functions of bounded type are dense in $H_b(U)$ for every balanced open set $U$ (see \cite[Theorem 7.11]{Muj86}), proceeding  as in the previous proof we also have the following.
 \begin{theorem}\label{CVT_HbU}
   The Cluster Value Theorem holds for $H_b(X)$ if and only if it holds for $H_b(U)$ for every (or some) balanced open set $U\subset X$.
 \end{theorem}
 Clearly, for the \emph{if} part of the theorem we do need $U$ to be balanced.

\section{Applications of the Cluster Value Theorem}\label{aplicaciones}

The Cluster Value Theorem  gives information about the structure of the spectrum of the algebras in which it holds. In this section we present some results and consequences in this setting.

\subsection{Fibers of $\boldsymbol{\M(A_u(B_X))}$}

We apply our results to describe the fibers of $\M(A_u(B_X))$ over points of the sphere. We know that for every $x\in \overline B_X$, the evaluation $\delta_x$ belongs to the fiber $\M_x(A_u(B_X))$ over $x$. When there is no other element, we say that the fiber is trivial.
Recall that a point $x\in \overline B_X$ is a \emph{point of continuity}  if any net $(x_\alpha)\subset \overline B_X$ weakly converging to $x$ also converges in norm.
We denote by $S_X$ the unit sphere of $X$ (i.e., the set of unit norm vectors in $X$).

\begin{proposition}\label{PC}
 Let $X$ be a Banach space whose dual has the bounded approximation property. If $x\in S_X$ is a point of continuity, then the fiber of $\mathcal M(A_u(B_X))$ over $x$ is trivial.
\end{proposition}
\begin{proof}
 Let $\varphi\in\mathcal M_{x}(A_u(B_X))$ and $f\in  A_u(B_X)$. By the Cluster Value Theorem~\ref{main1}, there is a net $(x_\alpha)\subset \overline B_X$ weakly converging to $x$ such that $\varphi(f)=\lim f(x_\alpha)$. Since $x$ is a point of continuity, $(x_\alpha)$ converges in norm to $x$ and therefore $\varphi(f)=\lim f(x_\alpha)=f(x)$.
\end{proof}

Recall that a Banach space has the \emph{ Kadec property } if every point in the sphere $S_X$ is point of continuity.
\begin{corollary}
 Let $X$ be a Banach space with the Kadec property whose dual has the bounded approximation property. Then the fiber of $\mathcal M(A_u(B_X))$ over every point of $S_X$ is trivial.
\end{corollary}

The previous proposition provides us with the following examples of spaces $X$ for which the fibers of $\mathcal M(A_u(B_X))$ over points in $S_{X}$ are trivial:
\begin{itemize}

\item if $X$  is locally uniformly rotund (it is easy to check that this implies the Kadec property) and  $X^{*}$ has the bounded approximation property;

\item if $X=d(w,1)$, where  $w=(w_k)_{k\in\mathbb N}$ is a non-increasing sequence of positive numbers so that $\sum_{k\in\mathbb N}w_k= +\infty$ \cite[Proposition 4]{molto1994quasi} (since in that proposition it is not assumed that $w_k\to 0$, we have $X=\ell_1$ as a particular case);

\item if $X=\ell_{M}$ is an Orlicz space (see \cite[Chapter~4]{LinTza73} for the definitions) for which both $M$ and $M^{*}$ satisfy condition $\Delta_{2}$ (this gives that the dual space has the bounded approximation property \cite{LinTza73})  and the Boyd index $\beta_{M} <\infty$ (this implies that the space is asymptotic uniformly convex \cite[Theorem~1.3.11]{BoMa10}, and therefore has the Kadec property). 

\end{itemize}
The result for $d(w,1)$ (and $\ell_1$) is equivalent to \cite[Corollary 2.4]{AroCarLasMae}, where the proof is based on the fact that points on the sphere of these spaces are strong peak points for the uniform algebra generated by  linear functionals. This last claim can be deduced from the proof of  \cite[Theorems 2.4 and 2.6]{AcoLou07}. We remark that, since uniformly convex spaces are locally uniformly rotund, we can conclude that uniformly convex spaces whose duals have the bounded approximation property have trivial fibers on the sphere. However, this is essentially shown in \cite[Corollary 2.3]{AroCarLasMae} without assuming the bounded approximation property for the dual.

\bigskip
By a classical result of Kadec, every separable Banach space can be renormed to be locally uniformly rotund \cite[Theorem 2.6]{DeGoZi93}. As a consequence, any separable Banach space whose dual has the bounded approximation property (which is invariant under renorming), can be renormed so that every fiber of $\mathcal M_{x}(A_u(B_X))$ over points of $S_X$ are trivial. In the opposite direction, we have the following proposition, which shows that the property of having trivial fibers is an isometric condition.

\begin{proposition}\label{fibrarenormada}
 Let $X$ be a Banach space that admits a homogeneous polynomial that is not weakly continuous on bounded sets and let $e\in S_X$. Then there is a renorming $Y=(X,\||\cdot\||)$ such that $e\in S_Y$ and such that the fiber over $e$ in $\mathcal M(A_u(B_Y))$ is not trivial. Moreover the cardinality of $\mathcal M_e(A_u(B_Y))$ is at least $c$.
\end{proposition}
\begin{proof}
Taking, if necessary, the composition of the polynomial of the statement with an affine map, we can take a polynomial $P$ on $X$ and a net $(x_\alpha)\subset X$, with $\|x_\alpha\|< 1/3$ for every $\alpha$, weakly converging to $0$ with $P(x_\alpha)\nrightarrow 0$.
Let $e\in S_{X}$ and $x^{*} \in \overline{ B_{X^*}}$ such that $x^{*}(e)=1$. . At least one of the homogeneous parts (name it $Q$ and suppose it is $k$-homogeneous) of the non-homogeneous polynomial $P(\cdot-e)$ is not weakly continuous at $e$ and  satisfies moreover, after taking a subnet,
$$|Q(e+x_\alpha) - Q(e)|> \varepsilon,$$
for some $\varepsilon>0$ and every $\alpha$.

The space  $X$ is isomorphic to $Y=\langle e\rangle \oplus_\infty Ker(x^{*})$. Let $y_\alpha=e+(x_\alpha-x^{*}(x_\alpha)e)$. Note that $x_\alpha-x^{*}(x_\alpha)e$ belongs to $B_{X}\cap{Ker(x^{*})}$. Then  $(y_\alpha)\subset \overline B_{Y}$, converges weakly to  $e\in \overline B_{Y}$ but
 $$
 | Q(y_\alpha)-  Q(e)|= \Big| Q(e+x_\alpha) +\sum_{j=1}^k\binom{k}{j}\ (- x^{*}(x_\alpha))^j\ \v{ Q}((e+x_\alpha)^{k-j},e^j)-  Q(e)\Big| > \varepsilon,
 $$
 for $\alpha>\alpha_0$ (where $\v{ Q}$ denotes the symmetric $k$-linear form associated to $Q$).  By compactness, $(\delta_{y_\alpha})$ has an accumulation point $\varphi\in\mathcal M(A_u(B_Y))$ that satisfies that $\varphi(Q)\ne \delta_e(Q)$ and $\pi(\varphi)=e$.

 Moreover, since $Cl(Q,e)$ is compact and connected in $\mathbb C$ (see \cite[page 2357]{AroCarLasMae})  and has at least two elements ($Q(e)$ and $\varphi(Q)$), then $card(Cl(Q,e))\ge c$. Since $Cl(Q,e)\subset \hat Q(\mathcal M_e(A_u(B_Y))$ the final comment follows.
\end{proof}

Note also that under a similar hypothesis it was proved in \cite[Proposition 2.6]{AroCarLasMae} (see also  \cite[Theorem 3.1]{aronoberwolfach}) that all fibers over the interior points of $\overline{B_{X^{**}}}$ are non-trivial.
As stated in \cite[Remark 3.2]{aronoberwolfach}, this assumption is necessary and, in some sense, sufficient:  If the restriction of every polynomial to $B_X$ is weakly continuous and $X^*$ has the approximation property, reasoning as in \cite[Theorem 7.2]{aron1995weak} we have
$\mathcal M_z(A_u(B_X)) = \{\delta_z\}$ for all $z \in \overline B_{X^{**}}$. It is unknown if the same happens if $X^{*}$ does not have the approximation property.
It should also be mentioned that in  most Banach spaces we can find  homogeneous polynomials which are not weakly continuous on bounded sets (as in the hypothesis of Proposition \ref{fibrarenormada}), see for example \cite{Din99}.

\subsection{The spectrum of $H_b(U)$}

Now we present some direct consequences of the Cluster Value Theorem for the algebra $H_b(U)$.  In \cite[Proposition 18]{CarGarMae05} it was proved that for any open set $U\subset X$ we have $$\bigcup_n \overline{U_n}^{w^*}\subset \pi(\M(H_b(U))),$$ where $(U_n)_n$ is a sequence of fundamental $U$-bounded sets, and that equality holds when $U$ and the sets $U_n$ are absolutely convex. We now show that we also have an equality whenever the Cluster Value Theorem holds for $H_b(U)$.
 \begin{corollary}
   Let $X$ be Banach and let $U\subset X$ be an open subset such that the \emph{Cluster Value Theorem} holds for $H_b(U)$. Then
   \begin{itemize}
    \item[i)] If $\varphi\in \M(H_b(U))$ and $\varphi\prec A$ then $\pi(\varphi)\in\overline{A}^{w^*}$.
    \item[ii)] $\pi(\M(H_b(U)))=\bigcup_n \overline{U_n}^{w^*}$, where $(U_n)_n$ is any sequence of fundamental $U$-bounded sets.
   \end{itemize}
  In particular, the previous statements are true for any open and balanced subset of a Banach space whose dual  has the bounded approximation property.
  \end{corollary}
  \begin{proof}
   Since $Cl_A(f,\pi(\varphi))\ne \emptyset$ for every $f\in H_b(U)$, i) follows from the definition of the cluster set $Cl_A$. The second statement follows from the first one since any $\varphi\in \M(H_b(U))$ satisfies $\varphi\prec U_n$ for some $n$.
  \end{proof}

  If $U$ is a balanced open subset and $A$ is $U$-bounded, the polynomial hull of $A$ is defined as (see for example \cite[Section 3]{CarMur12})
  $$
  \hat A_{\mathcal P}=\{z\in X^{**}\,:\, |\overline P(z)|\le \|P\|_A,\textrm{ for every }P\in\mathcal P(X)\},
  $$
  where $\overline P$ stands for the Aron-Berner extension of $P$, see \cite{Din99}.
By the Hahn-Banach Theorem, it is easy to see that $\hat A_{\mathcal P}$ is contained in the $w^*$-closure of the absolutely convex hull of $A$. When the Cluster Value Theorem holds, the absolute convex hull is not necessary: since polynomials are dense in $H_b(U)$, we have $z\in\hat A_{\mathcal P}$ if and only if $\delta_z\prec A$, and then the above corollary implies that $\hat A_{\mathcal P}\subset \overline{A}^{w^*}$.

\bigskip

Finally, proceeding as in Proposition \ref{PC}, it can be proven that if $X$ is a Banach space whose dual has the bounded approximation property, $x$ is a point of continuity of $\|x\|B_X$, and $U\subset X$ is balanced and open then $\varphi\in \M(H_b(U))$ in the fiber of $x$ is an evaluation if and only if $\varphi\prec \|x\|B_X$, or equivalently, $\M_{x,\|x\|B_X}(H_b(U))=\{\delta_x\}$. This happens for every $x\in U$ if the space $X$ has the Kadec property and $X^*$ has the bounded approximation property.

\subsection*{Acknowledgments}
We would like to thank our friends L. C. Garc\'ia-Lirola, A. J. Guirao, V. Montesinos, A. P\'erez and M. Raja for drawing our attention to the Kadec property and for many helpful remarks on the subject. We also thank our friend Manolo Maestre for his useful suggestions and comments.


\begin{thebibliography}{10}

\bibitem{AcoLou07}
Mar{\'{\i}}a~D. Acosta and Mary~Lilian Louren{\c{c}}o.
\newblock Shilov boundary for holomorphic functions on some classical {B}anach
  spaces.
\newblock {\em Studia Math.}, 179(1):27--39, 2007.

\bibitem{AlKa06}
Fernando Albiac and Nigel~J. Kalton.
\newblock {\em Topics in {B}anach space theory}, volume 233 of {\em Graduate
  Texts in Mathematics}.
\newblock Springer, New York, 2006.

\bibitem{arens1958dense}
Richard Arens.
\newblock Dense inverse limit rings.
\newblock {\em The Michigan Mathematical Journal}, 5(2):169--182, 1958.

\bibitem{aronoberwolfach}
Richard Aron, Domingo Garc\'ia, Javier Falc\'o, and Manuel Maestre.
\newblock Analytic structure in fibers.
\newblock {\em Oberwolfach Preprints}.

\bibitem{AroCarGamLasMae12}
Richard~M Aron, Daniel Carando, TW~Gamelin, Silvia Lassalle, and Manuel
  Maestre.
\newblock Cluster values of analytic functions on a {B}anach space.
\newblock {\em Mathematische Annalen}, 353(2):293--303, 2012.

\bibitem{AroCarLasMae}
Richard~M Aron, Daniel Carando, Silvia Lassalle, and Manuel Maestre.
\newblock Cluster values of holomorphic functions of bounded type.
\newblock {\em Trans. Amer. Math. Soc.}, 368:2355--2369, 2016.

\bibitem{AroColGam91}
Richard~M. Aron, Brian~J. Cole, and Theodore~W. Gamelin.
\newblock {Spectra of algebras of analytic functions on a {B}anach space.}
\newblock {\em J. Reine Angew. Math.}, 415:51--93, 1991.

\bibitem{aron1995weak}
Richard~M Aron, Brian~J Cole, and Theodore~W Gamelin.
\newblock Weak--star continuous analytic functions.
\newblock {\em Canadian Journal of Mathematics}, 47(4):673--683, 1995.

\bibitem{BoMa10}
Laetitia Borel-Mathurin.
\newblock {\em Isomorphismes non lin{\'e}aires entre espaces de Banach}.
\newblock PhD thesis, Paris 6, 2010.

\bibitem{CarGarMae05}
Daniel Carando, Domingo Garc\'{\i}a, and Manuel Maestre.
\newblock {Homomorphisms and composition operators on algebras of analytic
  functions of bounded type.}
\newblock {\em Adv. Math.}, 197(2):607--629, 2005.

\bibitem{carando2012spectra}
Daniel Carando, Domingo Garc{\'\i}a, Manuel Maestre, and Pablo Sevilla-Peris.
\newblock On the spectra of algebras of analytic functions.
\newblock In {\em Topics in complex analysis and operator theory}, pages
  165--198, 2012.

\bibitem{CarMur12}
Daniel Carando and Santiago Muro.
\newblock {Envelopes of holomorphy and extension of functions of bounded type.}
\newblock {\em Adv. Math.}, 229(3):2098--2121, 2012.

\bibitem{Carleson62}
Lennart Carleson.
\newblock Interpolations by bounded analytic functions and the corona problem.
\newblock {\em Annals of Mathematics}, pages 547--559, 1962.

\bibitem{casazza2001approximation}
Peter~G Casazza.
\newblock Approximation properties.
\newblock {\em Handbook of the geometry of {B}anach spaces}, 1:271--316, 2001.

\bibitem{DeGoZi93}
Robert Deville, Gilles Godefroy, and V{\'a}clav Zizler.
\newblock {\em Smoothness and renormings in {B}anach spaces}, volume~64 of {\em
  Pitman Monographs and Surveys in Pure and Applied Mathematics}.
\newblock Longman Scientific \& Technical, Harlow; copublished in the United
  States with John Wiley \& Sons, Inc., New York, 1993.

\bibitem{Din99}
Se\'{a}n Dineen.
\newblock {\em {Complex analysis on infinite dimensional spaces.}}
\newblock {Springer Monographs in Mathematics. London: Springer}, 1999.

\bibitem{johnson2015cluster}
William Johnson and Sofia Ortega~Castillo.
\newblock The cluster value problem in spaces of continuous functions.
\newblock {\em Proceedings of the American Mathematical Society},
  143(4):1559--1568, 2015.

\bibitem{johnson2014cluster}
William~B Johnson and Sofia~Ortega Castillo.
\newblock The cluster value problem for {B}anach spaces.
\newblock {\em Illinois Journal of Mathematics}, 58(2):405--412, 2014.

\bibitem{LinTza73}
Joram Lindenstrauss and Lior Tzafriri.
\newblock {\em Classical {B}anach spaces}.
\newblock Lecture Notes in Mathematics, Vol. 338. Springer-Verlag, Berlin,
  1973.

\bibitem{molto1994quasi}
An\'ibal Molt{\'o}, Vicente Montesinos, and Stanimir Troyanski.
\newblock On quasi-denting points, denting faces and the geometry of the unit
  ball of $d (w, 1)$.
\newblock {\em Archiv der Mathematik}, 63(1):45--55, 1994.

\bibitem{Muj86}
Jorge Mujica.
\newblock {\em {Complex analysis in {B}anach spaces. Holomorphic functions and
  domains of holomorphy in finite and infinite dimensions.}}
\newblock {North-Holland Mathematics Studies, 120. Notas de Matem\'atica, 107.
  Amsterdam/New York/Oxford: North-Holland. XI}, 1986.

\bibitem{Schark61}
I.J. Schark.
\newblock Maximal ideals in an algebra of bonded analytic functions.
\newblock {\em Journal of Mathematics and Mechanics}, 10(5):735--746, 1961.

\bibitem{sibony1974prolongement}
Nessim Sibony.
\newblock Prolongement analytique des fonctions holomorphes born{\'e}es.
\newblock In {\em S{\'e}minaire Pierre Lelong (Analyse) Ann{\'e}e 1972--1973},
  pages 44--66. Springer, 1974.

\end{thebibliography}
\end{document}